\documentclass{article}
\usepackage{graphicx}
\usepackage{booktabs}
\usepackage{siunitx}
\usepackage{amsfonts, amssymb, amsmath, amsthm}
\usepackage{float}

\newtheorem{theorem}{Theorem}[section]
\newtheorem{corollary}[theorem]{Corollary}
\newtheorem{conjecture}[theorem]{Conjecture}
\newtheorem{definition}[theorem]{Definition}
\newtheorem{fact}[theorem]{Fact}

\providecommand{\keywords}[1]
{
  \small	
  \textbf{Keywords } #1
}

\title{Spectral Analysis of the $D_{\log}^{(\lambda, N)}$ Operators}
\author{Dominik Śliwiński}
\date{January 2026}

\begin{document}

\maketitle
\begin{abstract}
This paper investigates the recent Connes-Consani-Moscovici $D_{\log}^{(\lambda, N)}$ operators, whose spectra are currently hypothesized to approach the zeros of $\zeta\left(\frac{1}{2} +is\right)$ as $\lambda, N \rightarrow \infty$. It turns out that when considering different standard notions of error, the dissonance between the spectra and Riemann $\zeta$ zeros either appears to or can be proven to be inverse logarithmic in nature, which elegantly fits the distribution of prime numbers. \\ \\
\end{abstract}
\section{Introduction}
By the recent works of Connes, Consani and Moscovici \cite{CCM25}, we are aware that it is possible to construct a family of operators \smash{$D_{\log}^{(\lambda, N)}$} satisfying the rank-one constraint \smash{$D_{\log}^{(\lambda, N)}(\xi) = 0$}, where $\xi$ is an even eigenvalue corresponding to an even simple Weil quadratic form, which forces the eigenvalues of any of the operators to align with zeros of $\zeta(\frac{1}{2} + is)$ up to the resolution that the finite-dimensional Hilbert space that the operator acts on offers. In the next section, the idea of \smash{$D_{\log}^{(\lambda, N)}$} will be recalled, and basic ideas of error, such as the Mean Absolute Error $\epsilon$, and the $L^{\infty}$ norm equivalent $\mathcal{E}$ will be defined. In a later section, an important property of $\epsilon$ will be proved, namely that

\[ \epsilon(\lambda, N) \geq \frac{1}{4\ln \lambda} \]
i.e. the Mean Absolute Error has an inverse logarithmic nature. Furthermore, when $\lambda$ is scaled to encompass information about $N$ primes in a certain way, we WILL see that
\[ \epsilon(N) = \Omega\left(\frac{1}{\ln N}\right) \]
Finally, numerical evidence for the similar inverse logarithmic nature of $\mathcal{E}$ will be presented, together with other computational results related to this family of operators and its relation to Riemann $\zeta$ zeros.

\section{Preliminaries}
\subsection{The $\mathcal{H}_\lambda$ space and the $D_{\log}^{(\lambda, N)}$ operator}
We define the space on which the operators in this paper will mostly act.
\begin{definition}
$\mathcal{H}_\lambda$ is defined as the space
\[ \mathcal{H}_\lambda := L^2([\lambda^{-1}, \lambda], d^{\ast}u)\]
where $d^{\ast}u$ denotes the multiplicative Haar measure.
\end{definition}
\noindent As we will be working pretty much exclusively with real eigenvalues, a logarithmic substitution could be made which would yield an isometric isomorphism to the additive space $L^2([- \ln \lambda, \ln \lambda], du)$, which may be an easier way to conceptualize certain ideas for some.
\begin{definition}
Within $\mathcal{H}_\lambda$, we define the scaling generator as the following differential operator:
\[ \hat{D} = -i \frac{d}{du} \]
\end{definition}
\noindent We will also use the following useful observation \cite{Folland97}.
\begin{fact}
The scaling generator $\hat{D}$ famously satisfies the canonical commutation relation $[\hat{u}, \hat{D}] = iI$, where $\hat{u}$ is the position operator under the $d^{\ast}u$ measure.
\end{fact}
\noindent Below is a definition of the $D_{\log}^{(\lambda, N)}$ operator family. A thorough reading of the original Connes-Consani-Moscovici article \cite{CCM25} is highly recommended.
\begin{definition}
The operator $D_{\log}^{(\lambda, N)}$ is defined as the compression of the scaling generator $\hat{D}$ onto a finite-dimensional subspace $E_N \subset \mathcal{H}_\lambda$. This subspace $E_N$ is $(2N+1)$-dimensional and localized within a logarithmic window of total length $L = 2 \ln \lambda$.  As established in the Connes-Consani-Moscovici paper, this is the unique operator which satisfies a unique rank-one constraint
\[ D_{\log}^{(\lambda, N)}(\xi) = 0 \]
where $\xi \in E_N$ is the vector corresponding to the truncated Weil quadratic form, and agrees on the Dirichlet kernel with $D_{\log}^\lambda$, which is defined analogously but is not restrictive to $E_n$.

\end{definition}

\subsection{The errors $\epsilon(\lambda, N)$ and $\mathcal{E}(\lambda, N)$}
It is hypothesized that the spectrum of $D_{\log}^{(\lambda, N)}$ gets arbitrarily close to the zeros of $\zeta\left(\frac{1}{2} + is\right)$. To measure the exact information about how close the spectrum actually gets to aforementioned zeros, multiple notions of dissonance or ``error'' will be defined. Since $E_N$ is a $(2N+1)$-dimensional subspace, the operator $D_{\log}^{(\lambda, N)}$ has exactly $2N+1$ eigenvalues. By the self-adjointness of the scaling generator and the symmetry of the truncated Weil quadratic form \cite{Weil52}, the spectrum is symmetric about the origin. Let the set of eigenvalues be denoted by:
\[ \sigma\left(D_{\log}^{(\lambda, N)}\right) = \{ \nu_{-N}, \dots, \nu_{-1}, \nu_0, \nu_1, \dots, \nu_N \} \]
where, by the rank-one constraint, $\nu_0 = 0$. We order the positive eigenvalues such that $0 < \nu_1 \leq \nu_2 \leq \dots \leq \nu_N$.  To quantify the convergence accuracy, we compare these values to the set of nontrivial Riemann zeros $\mathcal{Z} = \{ \frac{1}{2} \pm i\zeta_k \}_{k \in \mathbb{N}}$ (to not assume the Riemann Hypothesis, $\zeta_k$ may be complex). We order the zeros by magnitude of the imaginary part.

\begin{definition}
The Mean Absolute Error, denoted by $\epsilon(\lambda, N)$, is defined as the average distance between the $N$ positive eigenvalues of the operator and the first $N$ imaginary parts of the Riemann zeros.
\[ \epsilon(\lambda, N) := \frac{1}{N} \sum_{k=1}^N |\nu_k - \zeta_k| \]
\end{definition}

\begin{definition}
The Uniform Error (or $L^\infty$ error), denoted by $\mathcal{E}(\lambda, N)$, is defined as the maximum discrepancy between the corresponding pairs of eigenvalues and $\zeta$ zeros.
\[ \mathcal{E}(\lambda, N) := \max_{k \in \{1, \dots, N\}} |\nu_k - \zeta_k| \]
\end{definition}

\section{Inverse Logarithmic Lower Bound of $\epsilon$}
Let us consider the following theorem, which establishes a rigorous form of the inverse logarithmic nature of the dissonance between the spectrum of $D_{\log}^{(\lambda, N)}$ and Riemann $\zeta$ zeros.

\begin{theorem}
The error $\epsilon(\lambda, N)$ satisfies the weak lower bound inequality
\[ \epsilon(\lambda, N) \geq \frac{1}{4 \ln \lambda} \]
\end{theorem}

\begin{proof}
Let $\hat{u}$ be the position operator under the $d^{\ast}u$ measure, and $\hat{D} = -i\frac{d}{du}$ be the scaling generator. By Fact 2.3, $[\hat{u}, \hat{D}] = iI$. Therefore, by non-zero commutator, for any vector $\psi$ in the subspace $E_N \subset L^2([\lambda^{-1}, \lambda], d^{\ast}u)$,  the Heisenberg-Pauli-Weyl \cite{Folland97, Weyl50} inequality for $L^2$ gives
\[ \sigma_u \sigma_\zeta \ge \frac{1}{2} \]
where $\sigma_u^2 = \langle \psi, \hat{u}^2 \psi \rangle - \langle \psi, \hat{u} \psi \rangle^2$ and $\sigma_\zeta^2$ is the variance of the operator $D_{\log}^{(\lambda, N)}$, because it's a compression of $\hat{D}$ onto $E_N$. By the property $D_{log}^{(\lambda, N)}(\xi) = 0$ it has eigenvalues aligned with $\zeta_k$ up to at most the resolution of the Hilbert space $E_N$ with $\xi$ encoding the truncated Weil distribution. The maximum value of the position variance is bounded by the size of the space. $\mathcal{H}_\lambda$, which is a superspace of $E_N$, is under $d^{\ast}u$ and thus its size is $\ln \lambda - (-\ln \lambda) = 2\ln \lambda$ (one might recall the isomorphic additive space $L^2([- \ln \lambda, \ln \lambda], du)$). This brings us to the inequality
\[ \sigma_u \le 2 \ln \lambda \]
which combined with the Heisenberg-Pauli-Weyl inequality yields the spectral deviation inequality
\[ \sigma_\zeta \ge \frac{1}{4 \ln \lambda} \]
As $\epsilon(\lambda, N)$ is clearly bounded below by the spectral uncertainty $\sigma_\zeta$ of the truncated operator, we have
\[ \epsilon(\lambda,N) \ge \frac{1}{4 \ln \lambda} \]
\end{proof}

\begin{corollary}
Furthermore, when $\lambda$ is scaled to contain the information about $N$ primes in the way that $N \ln \lambda \sim p_N$, it is true that
\[ \epsilon(N) = \Omega\left(\frac{1}{\ln N}\right) \]
\end{corollary}
\begin{proof}
If $N \ln \lambda \sim p_N$, then by the Prime Number theorem \cite{Titchmarsh86}, $N \ln \lambda \sim N \ln N$ and thus $\ln \lambda \sim \ln N$ which in conjunction with Theorem 3.1 yields the above result.
\end{proof}

\section{Numerical Evidence for a Stronger Statement}
Based on an explicit calculation of the Weil quadratic form and using a parallel algorithm for calculating eigenvalues of self-adjoint operators represented by matrices, numerical computations were carried out on a graphics processing unit. The numerical accuracy was low with around 7 digits of decimal precision, but the goal was not to get exact computations, but rather to observe the general asymptotic nature of the spectra. A single argument $\kappa$ was introduced, such that $\kappa = N= \lambda$, for the simple empirical reason that this seemed to give good numerical results. Tests were done from $\kappa = 50$ up to $\kappa = 7500$ with a step of $50$. Near every multiple of $50$ in that range, a scan was performed for near points in order to find a nearby pair of points, such that one achieves a very loosely and numerically understood ``local maximum'' and the other one is a ``local minimum'' of $\mathcal{E}(\kappa)$ in the sense that it was numerically found to be the greater or lowest value $\mathcal{E}(\kappa)$ around that point. As it turns out, these are relatively close to each other. The test was performed against the first $1000$ Riemann $\zeta$ zeros. It appears to be likely that $\mathcal{E}(\kappa)$ is also of an inverse logarithmic nature.
\begin{figure}[H]
    \centering
    \includegraphics[scale=0.11]{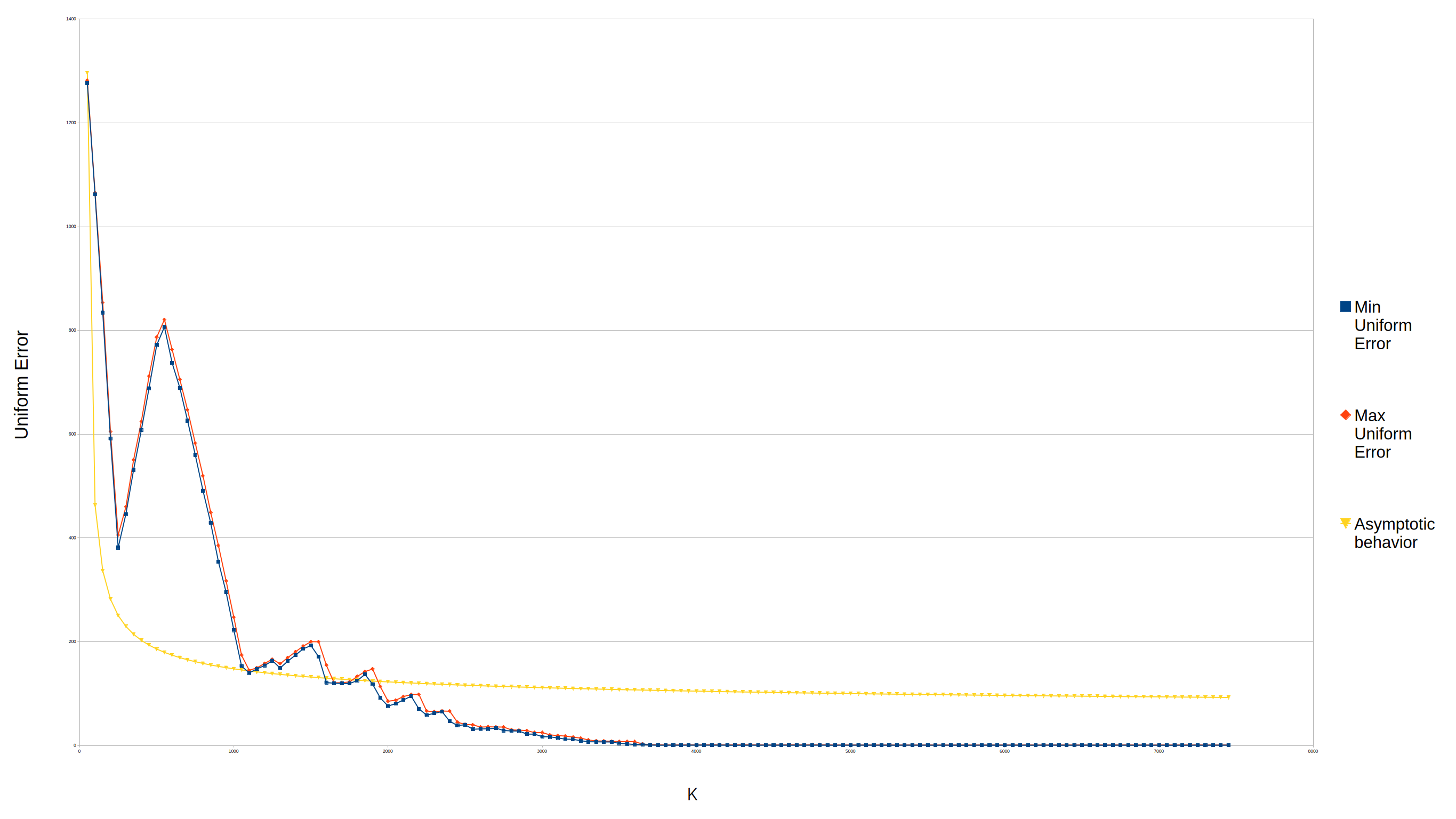}
    \caption{Minimum and maximum values of $\mathcal{E}(\kappa)$ around points consisting of multiples of $50$ from $50$ to $7500$ and an example logarithmic function showing the general nature of the error.}
\end{figure}
\vspace{7em}
\begin{figure}[H]
    \centering
    \includegraphics[scale=0.19]{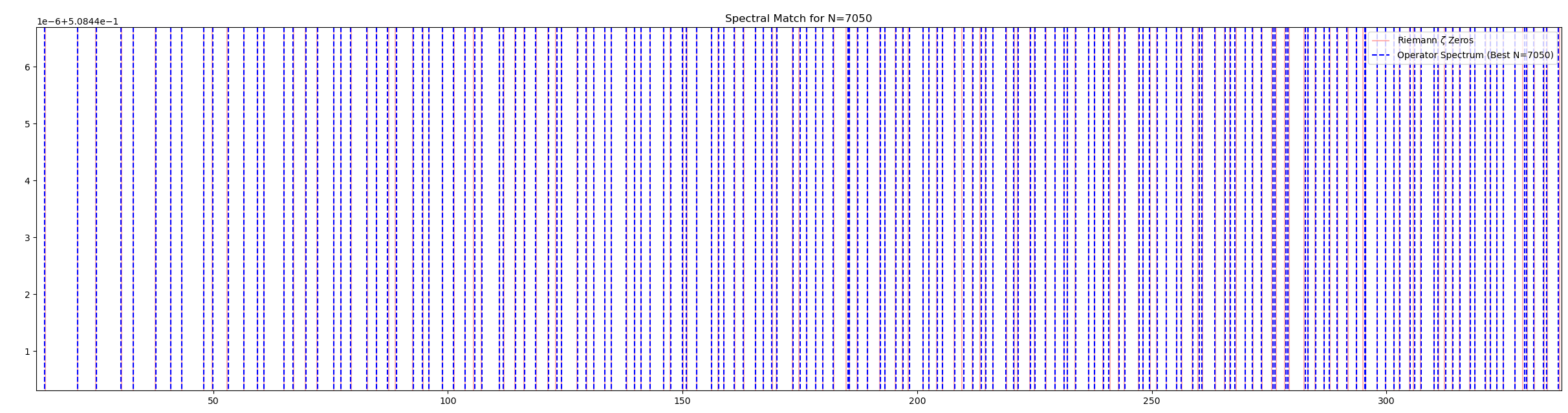}
    \caption{Visual representation of the highly matching behavior of the first zeros of $\zeta\left(\frac{1}{2}+is\right)$ and eigenvalues of $D_{\log}^{(7050, 7050)}$. The red lines are the imaginary parts of the Riemann $\zeta$ zeros and the dotted blue lines are the eigenvalues.}
\end{figure}
\begin{figure}[H]
    \centering
    \begin{tabular}{c S[table-format=4.10] S[table-format=4.4] S[table-format=1.2e-1]}
        \toprule
        Index & {$\zeta$ zero} & {$D_{\log}^{(7050, 7050)}$ eigenvalue} & {$\mathcal{E}$} \\
        \midrule
        976 & 1391.8532004433 & 1391.4514 & 4.02e-1 \\
        977 & 1392.6440277886 & 1392.8618 & 2.18e-1 \\
        978 & 1393.4334017408 & 1393.3467 & 8.67e-2 \\
        979 & 1394.8841846757 & 1394.7680 & 1.16e-1 \\
        980 & 1396.5441631237 & 1396.3391 & 2.05e-1 \\
        981 & 1397.8346233214 & 1397.9900 & 1.55e-1 \\
        982 & 1398.8376752014 & 1399.0261 & 1.88e-1 \\
        983 & 1399.8394729412 & 1399.7780 & 6.15e-2 \\
        984 & 1400.4269462974 & 1400.5144 & 8.75e-2 \\
        985 & 1402.5643472501 & 1402.5598 & 4.53e-3 \\
        986 & 1402.9737476409 & 1402.5598 & 4.14e-1 \\
        987 & 1404.0062921705 & 1403.8400 & 1.66e-1 \\
        988 & 1405.6669750593 & 1405.5550 & 1.12e-1 \\
        989 & 1407.0851427764 & 1406.9932 & 9.20e-2 \\
        990 & 1408.1363074962 & 1408.4056 & 2.69e-1 \\
        991 & 1409.3206810798 & 1409.1783 & 1.42e-1 \\
        992 & 1410.0248107258 & 1409.8967 & 1.28e-1 \\
        993 & 1411.2570568157 & 1411.2466 & 1.05e-2 \\
        994 & 1411.9656534618 & 1411.6276 & 3.38e-1 \\
        995 & 1413.8431487886 & 1413.9023 & 5.92e-2 \\
        996 & 1415.5857847955 & 1415.7133 & 1.27e-1 \\
        997 & 1415.7815813033 & 1415.7133 & 6.83e-2 \\
        998 & 1417.1028229338 & 1417.0852 & 1.76e-2 \\
        999 & 1418.6969638525 & 1418.3616 & 3.35e-1 \\
        1000 & 1419.4224809460 & 1419.7621 & 3.40e-1 \\
        \bottomrule
    \end{tabular}
    \caption{Comparison between the last $25$ out of $1000$ calculated $\zeta\left(\frac{1}{2} + is\right)$ zeros and the eigenvalues of $D_{\log}^{(7050, 7050)}$.}
\end{figure}
\noindent Based on the numerical evidence visually presented in Figure 1, and the nature of $\epsilon(\lambda, N)$ proven in Theorem 3.1, it is sensible to conjecture that $\mathcal{E}(\kappa)$ also has an inverse logarithmic nature.
\begin{conjecture}
The limit
\[ \lim_{\kappa \to \infty} \mathcal{E}(\kappa)\ln \kappa \]
exists.
\end{conjecture}
\noindent One could even hypothesize that the limit is equal to a non-zero constant, as in $\mathcal{E}(\kappa)$ has a similar convergence rate to $0$ as $\frac{1}{\ln \kappa}$ and not much faster or not much slower. Perhaps even that this constant is $1$, i.e. they're asymptotically equivalent. It is important to note however that even just Conjecture 4.1 implies the Riemann Hypothesis. That is because it implies that the uniform error goes to zero, in which case the eigenvalues would approach the $\zeta$ zeros pointwise. It is worth noting that this is not true for the Mean Absolute Error $\epsilon(\kappa)$, which could converge to zero and yet the eigenvalues could potentially not converge pointwise.


\begin{thebibliography}{99}

\bibitem{CCM25} 
Connes A., Consani C., Moscovici H. (2025), 
\textit{Zeta Spectral Triples}. 
arXiv preprint arXiv:2511.22755

\bibitem{Connes99} 
Connes, A. (1999),
\textit{Trace formula in noncommutative geometry and the zeros of the Riemann zeta function}.
Sel. math., New ser. 5, 29

\bibitem{Riemann1859} 
Riemann, B. (1859). 
\textit{Über die Anzahl der Primzahlen unter einer gegebenen Grösse}. 
Monatsberichte der Berliner Akademie.

\bibitem{Weil52} 
Weil, A. (1852). 
\textit{Sur les "formules explicites" de la théorie des nombres premiers}. 
Comm. Sém. Math. Univ. Lund [Medd. Lunds Univ. Mat. Sem.].

\bibitem{Folland97} 
Folland, G. B., and Sitaram, A. (1997). 
\textit{The uncertainty principle: a mathematical survey}. 
Journal of Fourier Analysis and Applications, 3(3), 207-238.

\bibitem{Weyl50} 
Weyl, H. (1950). 
\textit{The Theory of Groups and Quantum Mechanics}. 
Dover Publications.

\bibitem{Titchmarsh86} 
Titchmarsh, E. C. (1986). 
\textit{The Theory of the Riemann Zeta-function}. 
Oxford University Press.

\bibitem{Odlyzko87} 
Odlyzko, A. M. (1987).
\textit{On the distribution of spacings between zeros of the zeta function}. 
Mathematics of Computation, 48(177), 273-308.

\bibitem{Golub13} 
Golub, G. H., and Van Loan, C. F. (2013). 
\textit{Matrix Computations}. 
Johns Hopkins University Press.

\bibitem{Slepian61}
D. Slepian and H. Pollack (1961),
\textit{Prolate spheroidal wave functions, Fourier analysis and uncertainty}.
Bell Syst. Tech. J.

\end{thebibliography}
\end{document}